\documentclass[a4paper, 11pt]{amsart}

\usepackage[dvips]{graphicx}
\usepackage{amsmath,amsthm,amssymb,pinlabel}
\usepackage[all]{xy}
\usepackage[dvipsnames]{xcolor} \SelectTips{cm}{}

\allowdisplaybreaks

\usepackage{hyperref}
\hypersetup{colorlinks=true,citecolor=brown,linkcolor=brown,urlcolor=black,filecolor=black}

\theoremstyle{definition}
\newtheorem{dfn}{Definition}[section]
\newtheorem{rem}[dfn]{Remark}
\newtheorem{example}[dfn]{Example}

\theoremstyle{plain}
\newtheorem{thm}[dfn]{Theorem}
\newtheorem{prop}[dfn]{Proposition}
\newtheorem{lem}[dfn]{Lemma}
\newtheorem{cor}[dfn]{Corollary}

\makeatletter
    
    \@addtoreset{equation}{section}
  \makeatother

\begin{document}

\title{A note on Penner's cocycle on~the~fatgraph~complex}

\author{Yusuke Kuno}
\address{Department of Mathematics, Tsuda University,
2-1-1 Tsuda-machi, Kodaira-shi Tokyo 187-8577, Japan}
\email{kunotti@tsuda.ac.jp}

\author{Kae Takezawa}
\address{Department of Mathematics, Tsuda University,
2-1-1 Tsuda-machi, Kodaira-shi Tokyo 187-8577, Japan}
\email{m18ktake@gm.tsuda.ac.jp}

\date{}

\subjclass[2010]{}
\thanks{The first named author is supported by JSPS KAKENHI (no. 18K03308).}

\begin{abstract}
We study a $1$-cocycle on the fatgraph complex of a punctured surface introduced by Penner.
We present an explicit cobounding cochain for this cocycle, whose formula involves a summation over trivalent vertices of a trivalent fatgraph spine.
In a similar fashion, we express the symplectic form of the underlying surface of a given fatgraph spine.
\end{abstract}

\maketitle

\section{Introduction}

The fatgraph (or ribbon graph) complex $\widehat{\mathcal{G}} = \widehat{\mathcal{G}}(\Sigma)$ of a punctured oriented surface $\Sigma$ serves as a combinatorial model for the Teichm\"{u}ller space of $\Sigma$ and the action of the mapping class group $\mathcal{M} = \mathcal{M}(\Sigma)$ on it.
The cells of $\widehat{\mathcal{G}}$ are indexed by isotopy classes of fatgraph spines, which are graphs embedded in the surface satisfying certain conditions, and its face relations are described by contracting non-loop edges in fatgraph spines.

The complex $\widehat{\mathcal{G}}$ has been used to study the cohomology of the mapping class group and the moduli space of Riemann surfaces.
Among others, Morita and Penner \cite{MP} studied twisted first cohomology of the mapping class group by using $\widehat{\mathcal{G}}$.
More concretely, they considered a once punctured (or bordered) surface and ``lifted'' the extended first Johnson homomorphism
\[
\tilde{k} \in H^1(\mathcal{M}; \textstyle{\frac{1}{2}} \wedge^3 H)
\]
to a $1$-cocycle on $\widehat{\mathcal{G}}$.
Here, $H$ is the first integral homology group of $\Sigma$ and $\wedge^3 H$ is its third exterior power.
The key feature of \cite{MP} is that their cocycle
\[
j \in Z^1(\widehat{\mathcal{G}}; \wedge^3 H)
\]
takes a very explicit and simple form for flips (or Whitehead moves), which are local deformations of a trivalent fatgraph spine and which correspond to oriented $1$-cells of $\widehat{\mathcal{G}}$.
In this sense, the Morita-Penner construction is ``canonical''.

After the work of Morita and Penner, there have been known several variations and generalizations in other cohomological or topological objects related to the mapping class group; see \cite{ABP, ABMP, BKP, KPT, Ma}.
In each of them one can find interesting constructions making use of the combinatorics of fatgraph spines.

In this paper, based on the same line of study as above, we study invariants of trivalent fatgraph spines obtained by summation over trivalent vertices.
More concretely, we focus on a $1$-cocycle introduced by Penner in his textbook on decorated Teichm\"{u}ller theory \cite{PenBook},
\[
s \in Z^1(\widehat{\mathcal{G}}; S^2 (\wedge^2 H)),
\]
where $S^2(\wedge^2 H)$ is the second symmetric power of the second exterior power of $H$.
The coefficient of $s$ comes from a description of the target of the second Johnson homomorphism given in \cite{MorCasson, MorTaniguchi}.
As was shown in \cite{KPT}, the corresponding twisted cohomology class in $H^1(\mathcal{M}; S^2(\wedge^2 H))$ is trivial over the \emph{rational} coefficients, but this is simply because the cohomology group itself is trivial.
We will prove that this vanishing of the cohomology class more directly and over the \emph{integral} coefficients, by giving an explicit $0$-cochain
\[
\xi \in C^0(\widehat{\mathcal{G}}; S^2(\wedge^2 H))
\]
which cobounds $s$ and is $\mathcal{M}$-equivariant.
Since $0$-cells of $\widehat{\mathcal{G}}$ correspond to trivalent fatgraph spines, $\xi$ is nothing but an assignment
\[
\xi_G \in S^2(\wedge^2 H)
\]
to each trivalent fatgraph spine $G \subset \Sigma$.
The formula for $\xi_G$ is of the form
\[
\xi_G = \frac{1}{2} \sum_v \eta_v \cdot \eta_v,
\]
where the sum is taken over all trivalent vertices of $G$, and $\eta_v \in \wedge^2 H$ is defind by using the homology classes dual to the oriented edges pointing toward $v$. 

As a by-product of our study, we find a simple combinatorial formula for the symplectic form
\[
\omega \in \wedge^2 H
\]
of the underlying surface of a given fatgraph spine $G$.
The element $\omega$ is dual to the intersection form on $H$.
Our formula is of the form
\[
\omega = \frac{1}{2} \sum_v \eta_v.
\]

The results above will be stated more precisely and proved in \S \ref{sec:results}
(Theorem \ref{thm:main} and Theorem \ref{thm:symplectic_form}).

In order to simplify the exposition we restrict ourselves to the case of a once punctured surface, but the general case may be addressed similarly.

\subsection*{Acknowledgements}

We thank Gw\'ena\"el Massuyeau for pointing out a relation between our work and his paper \cite{Ma}.

\section{Fatgraph complex and Penner's cocycle} \label{sec:pre}

We collect basic materials about the fatgraph complex of a punctured surface, and recall the definition of Penner's cocycle.
For more details, see \cite{PenBook, KPT}.

Throughout the paper, we fix a positive integer $g$.

\subsection{Fatgraph spines}

Let $\Sigma = \Sigma^1_g$ be a closed oriented surface of genus $g$ together with a distinguished basepoint $*$.
We call $\Sigma$ a \emph{once punctured} surface of genus $g$.

By a graph, we mean a finite CW complex of dimension one.
Each edge of a graph admits two orientations.
An \emph{oriented edge} of a graph is an edge of the graph equipped with an orientation.
If $e$ is an oriented edge of a graph, we denote by $\overline{e}$ the oriented edge that is obtained by reversing the orientation of $e$.

\begin{dfn}
A \emph{fatgraph spine} of $\Sigma$ is a graph $G$ embedded in $\Sigma \setminus \{ * \}$ such that the inclusion map $G\hookrightarrow \Sigma \setminus \{ *\}$ is a homotpy equivalence and the valency of every vertex of $G$ is at least three.
\end{dfn}

\begin{figure}[ht]
\begin{minipage}[b]{0.49\hsize}
  \centering
{\unitlength 0.1in%
\begin{picture}(16.0000,12.0000)(6.0000,-18.0000)%
%
\special{pn 13}%
\special{ar 1400 1200 800 600 0.0000000 6.2831853}%
%
\special{pn 13}%
\special{ar 1400 1200 300 200 3.1415927 6.2831853}%
%
\special{pn 13}%
\special{ar 1400 1140 330 220 6.2831853 3.1415927}%
%
\special{pn 8}%
\special{ar 1400 1200 550 400 0.0000000 6.2831853}%
%
\special{pn 4}%
\special{sh 1}%
\special{ar 1010 920 16 16 0 6.2831853}%
%
\special{pn 4}%
\special{sh 1}%
\special{ar 1790 920 16 16 0 6.2831853}%
%
\special{pn 8}%
\special{ar 1240 920 228 94 1.4341953 3.1415927}%
%
\special{pn 8}%
\special{ar 1640 920 150 286 4.7123890 6.2831853}%
%
\special{pn 8}%
\special{pn 8}%
\special{pa 1640 636}%
\special{pa 1603 636}%
\special{pa 1572 637}%
\special{fp}%
\special{pa 1504 645}%
\special{pa 1500 646}%
\special{pa 1472 657}%
\special{pa 1448 672}%
\special{pa 1445 676}%
\special{fp}%
\special{pa 1412 736}%
\special{pa 1407 753}%
\special{pa 1401 787}%
\special{pa 1400 802}%
\special{fp}%
\special{pa 1396 871}%
\special{pa 1394 895}%
\special{pa 1389 927}%
\special{pa 1385 938}%
\special{fp}%
\special{pa 1344 990}%
\special{pa 1338 994}%
\special{pa 1307 1005}%
\special{pa 1280 1012}%
\special{fp}%
%
\special{pn 4}%
\special{sh 1}%
\special{ar 2080 1200 16 16 0 6.2831853}%
\put(20.4000,-13.2000){\makebox(0,0)[lb]{$*$}}%
%
\special{pn 8}%
\special{pa 1440 1600}%
\special{pa 1360 1580}%
\special{fp}%
\special{pa 1440 1600}%
\special{pa 1360 1620}%
\special{fp}%
%
\special{pn 8}%
\special{pa 1274 812}%
\special{pa 1192 809}%
\special{fp}%
\special{pa 1274 812}%
\special{pa 1200 848}%
\special{fp}%
%
\special{pn 8}%
\special{pa 1760 750}%
\special{pa 1765 832}%
\special{fp}%
\special{pa 1760 750}%
\special{pa 1803 820}%
\special{fp}%
\put(11.6000,-7.7200){\makebox(0,0)[lb]{$e_3$}}%
\put(18.3000,-8.7200){\makebox(0,0)[lb]{$e_2$}}%
\put(13.6000,-17.4200){\makebox(0,0)[lb]{$e_1$}}%
\end{picture}}%
\caption{Fatgraph spine.}
\label{fig:fgs}
\end{minipage}
\begin{minipage}[b]{0.49\hsize}
  \centering
{\unitlength 0.1in%
\begin{picture}(16.0000,16.7500)(18.0000,-29.2500)%
%
\special{pn 8}%
\special{pa 2200 1500}%
\special{pa 3000 1500}%
\special{pa 3400 2193}%
\special{pa 3000 2886}%
\special{pa 2200 2886}%
\special{pa 1800 2193}%
\special{pa 2200 1500}%
\special{pa 3000 1500}%
\special{fp}%
%
\special{pn 8}%
\special{pa 2600 1500}%
\special{pa 2520 1480}%
\special{fp}%
\special{pa 2600 1500}%
\special{pa 2520 1520}%
\special{fp}%
%
\special{pn 8}%
\special{pa 2560 2886}%
\special{pa 2640 2866}%
\special{fp}%
\special{pa 2560 2886}%
\special{pa 2640 2906}%
\special{fp}%
%
\special{pn 8}%
\special{pa 1977 2496}%
\special{pa 2035 2555}%
\special{fp}%
\special{pa 1977 2496}%
\special{pa 2000 2575}%
\special{fp}%
%
\special{pn 8}%
\special{pa 2033 1788}%
\special{pa 1976 1848}%
\special{fp}%
\special{pa 2033 1788}%
\special{pa 2010 1868}%
\special{fp}%
%
\special{pn 8}%
\special{pa 3239 1911}%
\special{pa 3216 1831}%
\special{fp}%
\special{pa 3239 1911}%
\special{pa 3182 1852}%
\special{fp}%
%
\special{pn 8}%
\special{pa 3189 2557}%
\special{pa 3249 2499}%
\special{fp}%
\special{pa 3189 2557}%
\special{pa 3214 2479}%
\special{fp}%
%
\special{pn 4}%
\special{sh 1}%
\special{ar 2600 2195 16 16 0 6.2831853}%
\put(25.5500,-23.9000){\makebox(0,0)[lb]{$*$}}%
\put(18.2500,-26.9500){\makebox(0,0)[lb]{$e_1$}}%
\put(18.0000,-17.9500){\makebox(0,0)[lb]{$e_2$}}%
\put(25.0000,-13.9500){\makebox(0,0)[lb]{$e_3$}}%
\put(33.0000,-18.9500){\makebox(0,0)[lb]{$\overline{e}_1$}}%
\put(33.0000,-26.4500){\makebox(0,0)[lb]{$\overline{e}_2$}}%
\put(25.0000,-30.7000){\makebox(0,0)[lb]{$\overline{e}_3$}}%
\end{picture}}%
\caption{Cut along $G$.}
\label{fig:fgscut}
\end{minipage}
\end{figure}

Figure \ref{fig:fgs} shows a fatgraph spine of a surface of genus one.
The condition that $G\hookrightarrow \Sigma \setminus \{ * \}$ is a homotopy equivalence implies that if we cut $\Sigma$ along $G$ we obtain a polygon with a puncture in its interior which corresponds to $*\in \Sigma$.
We give the polygon an orientation which matches that of $\Sigma$.
Going around the boundary of the polygon in the \emph{clockwise} manner, we traverse each oriented edge of $G$ exactly once.
This gives the set of oriented edges of $G$ a cyclic ordering, which we denote by $\prec$.
For instance, the polygon obtained from the fatgraph spine in Figure \ref{fig:fgs} is depicted in Figure \ref{fig:fgscut}, and we have
$e_1 \prec e_2 \prec e_3 \prec \overline{e}_1 \prec \overline{e}_2 \prec  \overline{e}_3 \prec e_1$.

For each vertex of a fatgraph spine of $\Sigma$, the set of oriented edges pointing toward the vertex is endowed with a cyclic ordering induced from the orientation of the tangent space of $\Sigma$ at the vertex.

\subsection{Fatgraph complex and flips}

The \emph{fatgraph complex} of $\Sigma$, denoted by $\widehat{\mathcal{G}} = \widehat{\mathcal{G}}(\Sigma)$, is a CW complex of dimension $4g-3$ whose cells are in one to one correspondence with isotopy classes of fatgraph spines of $\Sigma$.
Each $0$-cell of $\widehat{\mathcal{G}}$ corresponds to a trivalent fatgraph spine, i.e. a fatgraph spine with all the vertices being trivalent.
In general, $n$-dimensional cells of $\widehat{\mathcal{G}}$ correspond to fatgraph spines which can be obtained from a trivalent one by collapsing $n$ edges.

Each oriented $1$-cell of $\widehat{\mathcal{G}}$ corresponds to a \emph{flip} (or a \emph{Whitehead move}) between two trivalent fatgraph spines;
if $e$ is an edge of a trivalent fatgraph spine $G$, a flip along $e$, denoted by $W_e$, is to deform $G$ by collapsing and expanding $e$ in the natural way to obtain another trivalent fatgraph spine $G'$.
We use a shorthand notation
\[
G\overset{W_e}{\longrightarrow} G'
\]
for the flip $W_e$.
See Figure \ref{fig:flip}, where the orientation of the surface matches the trigonometric orientation of the plan.
Here, since there is a canonical bijective correspondence between the four oriented edges in $G$ adjacent to $e$ and the four oriented edges in $G'$ adjacent to $e'$, we use the same letter for each pair of corresponding edges.
Note that it can happen for instance that $a$ and $c$ have the same underlying edge.

\begin{figure}[ht]
  \centering
{\unitlength 0.1in%
\begin{picture}(32.7500,16.0000)(8.7500,-24.4800)%
%
\special{pn 8}%
\special{pa 1300 1600}%
\special{pa 1900 1600}%
\special{fp}%
%
\special{pn 8}%
\special{pa 1900 1600}%
\special{pa 2200 1300}%
\special{fp}%
%
\special{pn 8}%
\special{pa 1900 1600}%
\special{pa 2200 1900}%
\special{fp}%
%
\special{pn 8}%
\special{pa 1300 1600}%
\special{pa 1000 1300}%
\special{fp}%
%
\special{pn 8}%
\special{pa 1300 1600}%
\special{pa 1000 1900}%
\special{fp}%
%
\special{pn 8}%
\special{pa 3800 1900}%
\special{pa 3800 1300}%
\special{fp}%
%
\special{pn 8}%
\special{pa 3800 1300}%
\special{pa 3500 1000}%
\special{fp}%
%
\special{pn 8}%
\special{pa 3800 1300}%
\special{pa 4100 1000}%
\special{fp}%
%
\special{pn 8}%
\special{pa 3800 1900}%
\special{pa 3500 2200}%
\special{fp}%
%
\special{pn 8}%
\special{pa 3800 1900}%
\special{pa 4100 2200}%
\special{fp}%
%
\special{pn 13}%
\special{pa 2500 1600}%
\special{pa 3200 1600}%
\special{fp}%
\special{sh 1}%
\special{pa 3200 1600}%
\special{pa 3133 1580}%
\special{pa 3147 1600}%
\special{pa 3133 1620}%
\special{pa 3200 1600}%
\special{fp}%
\put(27.0000,-15.2500){\makebox(0,0)[lb]{$W_e$}}%
\put(15.0000,-26.0000){\makebox(0,0)[lb]{$G$}}%
\put(37.0000,-26.0000){\makebox(0,0)[lb]{$G'$}}%
\put(15.5000,-15.5000){\makebox(0,0)[lb]{$e$}}%
\put(38.7500,-16.0000){\makebox(0,0)[lb]{$e'$}}%
\put(22.2500,-20.0000){\makebox(0,0)[lb]{$a$}}%
\put(22.2500,-13.0000){\makebox(0,0)[lb]{$b$}}%
\put(8.7500,-13.0000){\makebox(0,0)[lb]{$c$}}%
\put(8.7500,-20.0000){\makebox(0,0)[lb]{$d$}}%
\put(41.5000,-23.0000){\makebox(0,0)[lb]{$a$}}%
\put(41.5000,-10.0000){\makebox(0,0)[lb]{$b$}}%
\put(33.7500,-10.0000){\makebox(0,0)[lb]{$c$}}%
\put(33.7500,-23.0000){\makebox(0,0)[lb]{$d$}}%
%
\special{pn 8}%
\special{pa 1180 1480}%
\special{pa 1140 1400}%
\special{fp}%
\special{pa 1180 1480}%
\special{pa 1100 1440}%
\special{fp}%
%
\special{pn 8}%
\special{pa 2020 1480}%
\special{pa 2060 1400}%
\special{fp}%
\special{pa 2020 1480}%
\special{pa 2100 1440}%
\special{fp}%
%
\special{pn 8}%
\special{pa 2020 1720}%
\special{pa 2060 1800}%
\special{fp}%
\special{pa 2020 1720}%
\special{pa 2100 1760}%
\special{fp}%
%
\special{pn 8}%
\special{pa 1180 1720}%
\special{pa 1140 1800}%
\special{fp}%
\special{pa 1180 1720}%
\special{pa 1100 1760}%
\special{fp}%
%
\special{pn 8}%
\special{pa 3940 1160}%
\special{pa 3980 1080}%
\special{fp}%
\special{pa 3940 1160}%
\special{pa 4020 1120}%
\special{fp}%
%
\special{pn 8}%
\special{pa 3660 1160}%
\special{pa 3620 1080}%
\special{fp}%
\special{pa 3660 1160}%
\special{pa 3580 1120}%
\special{fp}%
%
\special{pn 8}%
\special{pa 3660 2040}%
\special{pa 3620 2120}%
\special{fp}%
\special{pa 3660 2040}%
\special{pa 3580 2080}%
\special{fp}%
%
\special{pn 8}%
\special{pa 3940 2040}%
\special{pa 3980 2120}%
\special{fp}%
\special{pa 3940 2040}%
\special{pa 4020 2080}%
\special{fp}%
\end{picture}}%
\caption{Flip.}
\label{fig:flip}
\end{figure}

Flips can be composed as morphisms of the fundamental path groupoid of the CW complex $\widehat{\mathcal{G}}$.
We read composition of flips from right to left.
There are three types of relations satisfied by compositions of flips:
\begin{description}
\item[Involutivity relation] $W_{e'} W_e = {\rm id}$ in the notation of Figure \ref{fig:flip}.
\item[Commutativity relation] $W_{e_1} W_{e_2} = W_{e_2} W_{e_1}$ if $e_1$ and $e_2$ are edges sharing no vertices in a trivalent fatgraph spine.
\item[Pentagon relation] $W_{f_4} W_{g_3} W_{f_2} W_{g_1} W_{f} = {\rm id}$ in the notation of Figure \ref{fig:pentagon}.
\end{description}

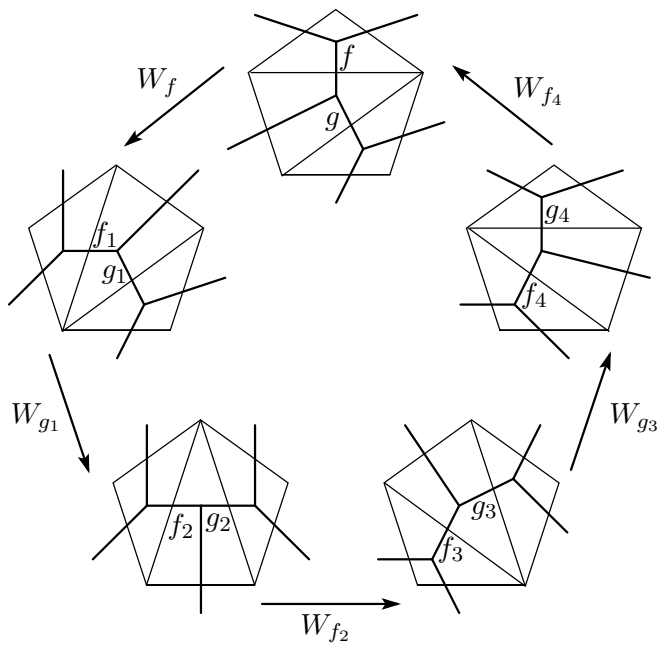
\begin{figure}[ht]
  \centering
{\unitlength 0.1in%
\begin{picture}(33.1800,31.5800)(7.9000,-34.5500)%
%
\special{pn 8}%
\special{pa 1502 3301}%
\special{pa 2062 3301}%
\special{pa 2235 2768}%
\special{pa 1782 2439}%
\special{pa 1329 2768}%
\special{pa 1502 3301}%
\special{pa 2062 3301}%
\special{fp}%
%
\special{pn 8}%
\special{pa 2062 3301}%
\special{pa 1782 2440}%
\special{fp}%
%
\special{pn 8}%
\special{pa 1502 3301}%
\special{pa 1782 2440}%
\special{fp}%
%
\special{pn 8}%
\special{pa 1065 1971}%
\special{pa 1625 1971}%
\special{pa 1798 1438}%
\special{pa 1345 1109}%
\special{pa 892 1438}%
\special{pa 1065 1971}%
\special{pa 1625 1971}%
\special{fp}%
%
\special{pn 8}%
\special{pa 1067 1977}%
\special{pa 1795 1438}%
\special{fp}%
%
\special{pn 8}%
\special{pa 1067 1977}%
\special{pa 1347 1116}%
\special{fp}%
%
\special{pn 8}%
\special{pa 3465 3301}%
\special{pa 2905 3301}%
\special{pa 2732 2768}%
\special{pa 3185 2439}%
\special{pa 3638 2768}%
\special{pa 3465 3301}%
\special{pa 2905 3301}%
\special{fp}%
%
\special{pn 8}%
\special{pa 3463 3307}%
\special{pa 2735 2768}%
\special{fp}%
%
\special{pn 8}%
\special{pa 3463 3307}%
\special{pa 3183 2446}%
\special{fp}%
%
\special{pn 8}%
\special{pa 2202 1159}%
\special{pa 2762 1159}%
\special{pa 2935 626}%
\special{pa 2482 297}%
\special{pa 2029 626}%
\special{pa 2202 1159}%
\special{pa 2762 1159}%
\special{fp}%
%
\special{pn 8}%
\special{pa 2932 626}%
\special{pa 2204 1165}%
\special{fp}%
%
\special{pn 8}%
\special{pa 3892 1971}%
\special{pa 3332 1971}%
\special{pa 3159 1438}%
\special{pa 3612 1109}%
\special{pa 4065 1438}%
\special{pa 3892 1971}%
\special{pa 3332 1971}%
\special{fp}%
%
\special{pn 8}%
\special{pa 4065 1438}%
\special{pa 3176 1438}%
\special{fp}%
%
\special{pn 8}%
\special{pa 3169 1438}%
\special{pa 3897 1977}%
\special{fp}%
%
\special{pn 13}%
\special{pa 1350 1557}%
\special{pa 1770 1137}%
\special{fp}%
\special{pa 1350 1557}%
\special{pa 1490 1837}%
\special{fp}%
\special{pa 1490 1837}%
\special{pa 1910 1697}%
\special{fp}%
\special{pa 1490 1837}%
\special{pa 1350 2117}%
\special{fp}%
\special{pa 1350 1557}%
\special{pa 1070 1557}%
\special{fp}%
\special{pa 1070 1557}%
\special{pa 790 1837}%
\special{fp}%
\special{pa 1070 1557}%
\special{pa 1070 1137}%
\special{fp}%
%
\special{pn 13}%
\special{pa 3121 2887}%
\special{pa 2841 2467}%
\special{fp}%
\special{pa 3121 2887}%
\special{pa 2981 3167}%
\special{fp}%
\special{pa 2981 3167}%
\special{pa 2701 3167}%
\special{fp}%
\special{pa 2981 3167}%
\special{pa 3121 3447}%
\special{fp}%
\special{pa 3121 2887}%
\special{pa 3401 2747}%
\special{fp}%
\special{pa 3401 2747}%
\special{pa 3541 2467}%
\special{fp}%
\special{pa 3401 2747}%
\special{pa 3681 3027}%
\special{fp}%
%
\special{pn 13}%
\special{pa 1784 2887}%
\special{pa 1784 3447}%
\special{fp}%
\special{pa 1784 2887}%
\special{pa 1504 2887}%
\special{fp}%
\special{pa 1504 2887}%
\special{pa 1504 2467}%
\special{fp}%
\special{pa 1504 2887}%
\special{pa 1224 3167}%
\special{fp}%
\special{pa 1784 2887}%
\special{pa 2064 2887}%
\special{fp}%
\special{pa 2064 2887}%
\special{pa 2344 3167}%
\special{fp}%
\special{pa 2064 2887}%
\special{pa 2064 2467}%
\special{fp}%
%
\special{pn 13}%
\special{pa 3548 1557}%
\special{pa 4108 1697}%
\special{fp}%
\special{pa 3548 1557}%
\special{pa 3548 1277}%
\special{fp}%
\special{pa 3548 1277}%
\special{pa 3268 1137}%
\special{fp}%
\special{pa 3548 1277}%
\special{pa 3968 1137}%
\special{fp}%
\special{pa 3548 1557}%
\special{pa 3408 1837}%
\special{fp}%
\special{pa 3408 1837}%
\special{pa 3688 2117}%
\special{fp}%
\special{pa 3408 1837}%
\special{pa 3128 1837}%
\special{fp}%
%
\special{pn 13}%
\special{pa 2484 745}%
\special{pa 1924 1025}%
\special{fp}%
\special{pa 2484 745}%
\special{pa 2484 465}%
\special{fp}%
\special{pa 2484 465}%
\special{pa 2064 325}%
\special{fp}%
\special{pa 2484 465}%
\special{pa 2904 325}%
\special{fp}%
\special{pa 2484 745}%
\special{pa 2624 1025}%
\special{fp}%
\special{pa 2624 1025}%
\special{pa 2484 1305}%
\special{fp}%
\special{pa 2624 1025}%
\special{pa 3044 885}%
\special{fp}%
%
\special{pn 8}%
\special{pa 2043 626}%
\special{pa 2925 626}%
\special{fp}%
\put(24.9700,-6.1500){\makebox(0,0)[lb]{$f$}}%
\put(24.2700,-9.2800){\makebox(0,0)[lb]{$g$}}%
\put(12.7000,-17.2800){\makebox(0,0)[lb]{$g_1$}}%
\put(12.1900,-15.4100){\makebox(0,0)[lb]{$f_1$}}%
\put(16.1000,-30.6100){\makebox(0,0)[lb]{$f_2$}}%
\put(18.0900,-30.2400){\makebox(0,0)[lb]{$g_2$}}%
\put(30.0800,-31.8100){\makebox(0,0)[lb]{$f_3$}}%
\put(31.8400,-29.5800){\makebox(0,0)[lb]{$g_3$}}%
\put(34.4000,-18.5700){\makebox(0,0)[lb]{$f_4$}}%
\put(35.6400,-14.1400){\makebox(0,0)[lb]{$g_4$}}%
%
\special{pn 13}%
\special{pa 1900 600}%
\special{pa 1400 1000}%
\special{fp}%
\special{sh 1}%
\special{pa 1400 1000}%
\special{pa 1465 974}%
\special{pa 1442 967}%
\special{pa 1440 943}%
\special{pa 1400 1000}%
\special{fp}%
\put(14.5000,-7.5000){\makebox(0,0)[lb]{$W_f$}}%
%
\special{pn 13}%
\special{pa 1000 2100}%
\special{pa 1200 2700}%
\special{fp}%
\special{sh 1}%
\special{pa 1200 2700}%
\special{pa 1198 2630}%
\special{pa 1183 2649}%
\special{pa 1160 2643}%
\special{pa 1200 2700}%
\special{fp}%
\put(8.0000,-25.0000){\makebox(0,0)[lb]{$W_{g_1}$}}%
%
\special{pn 13}%
\special{pa 2100 3400}%
\special{pa 2800 3400}%
\special{fp}%
\special{sh 1}%
\special{pa 2800 3400}%
\special{pa 2733 3380}%
\special{pa 2747 3400}%
\special{pa 2733 3420}%
\special{pa 2800 3400}%
\special{fp}%
\put(23.0000,-36.0000){\makebox(0,0)[lb]{$W_{f_2}$}}%
%
\special{pn 13}%
\special{pa 3700 2700}%
\special{pa 3900 2100}%
\special{fp}%
\special{sh 1}%
\special{pa 3900 2100}%
\special{pa 3860 2157}%
\special{pa 3883 2151}%
\special{pa 3898 2170}%
\special{pa 3900 2100}%
\special{fp}%
\put(39.0000,-25.0000){\makebox(0,0)[lb]{$W_{g_3}$}}%
%
\special{pn 13}%
\special{pa 3600 1000}%
\special{pa 3100 600}%
\special{fp}%
\special{sh 1}%
\special{pa 3100 600}%
\special{pa 3140 657}%
\special{pa 3142 633}%
\special{pa 3165 626}%
\special{pa 3100 600}%
\special{fp}%
\put(34.0000,-8.0000){\makebox(0,0)[lb]{$W_{f_4}$}}%
\end{picture}}%
\caption{Pentagon relation.}
\label{fig:pentagon}
\end{figure}

The involutivity relation comes from the fact that $W_{e'}$ is the oriented $1$-cell of $\widehat{\mathcal{G}}$ obtained by reversing the orientation of $W_e$.
The commutativity and pentagon relations come from the boundary of $2$-cells of $\widehat{\mathcal{G}}$.

There is a homotopy equivalence between the Teichm\"{u}ller space of the punctured surface $\Sigma$ and $\widehat{\mathcal{G}}$ which is equivariant under the action of the mapping class group of $\Sigma$.
In particular, the space $\widehat{\mathcal{G}}$ is contractible.
This implies the following basic facts:
\begin{itemize}
\item For any trivalent fatgraph spines $G$ and $G'$, there is a finite sequence of flips connecting $G$ to $G'$:
\[
G = G_0 \overset{W_1}{\longrightarrow} G_1 \overset{W_2}{\longrightarrow} \cdots \overset{W_m}{\longrightarrow} G_m = G'.
\]
\item Any two sequences of flips connecting $G$ to $G'$ are related by a finite application of the three types of relations above.
\end{itemize}

\subsection{Homology marking}

We denote by $H$ the first integral homology group of $\Sigma$.
Let $G$ be a fatgraph spine of $\Sigma$ and let $e$ be an oriented edge of $G$.
Then there is an oriented simple loop $\hat{e}$ in $\Sigma$ such that
$\hat{e}$ and $G$ have only one intersection which lies in the interior of $e$, and the ordered pair $(\hat{e},e)$ matches the positive frame for the tangent space of $\Sigma$ at the intersection;
see Figure \ref{fig:hmarking}.
The homotopy class of $\hat{e}$ is well defined.
We denote by
\[
\mu(e) \in H
\]
the homology class of $\hat{e}$ and call it the \emph{homology marking} of $e$.

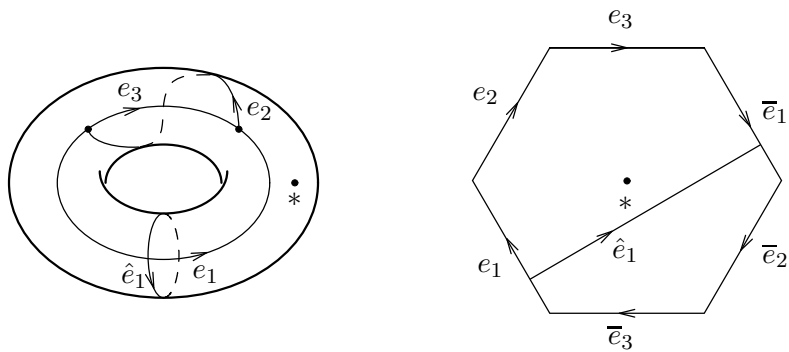
\begin{figure}[ht]
  \centering
{\unitlength 0.1in%
\begin{picture}(40.0000,16.7500)(6.0000,-19.1300)%
%
\special{pn 13}%
\special{ar 1400 1200 800 600 0.0000000 6.2831853}%
%
\special{pn 13}%
\special{ar 1400 1200 300 200 3.1415927 6.2831853}%
%
\special{pn 13}%
\special{ar 1400 1140 330 220 6.2831853 3.1415927}%
%
\special{pn 8}%
\special{ar 1400 1200 550 400 0.0000000 6.2831853}%
%
\special{pn 4}%
\special{sh 1}%
\special{ar 1010 920 16 16 0 6.2831853}%
%
\special{pn 4}%
\special{sh 1}%
\special{ar 1790 920 16 16 0 6.2831853}%
%
\special{pn 8}%
\special{ar 1240 920 228 94 1.4341953 3.1415927}%
%
\special{pn 8}%
\special{ar 1640 920 150 286 4.7123890 6.2831853}%
%
\special{pn 8}%
\special{pn 8}%
\special{pa 1640 636}%
\special{pa 1603 636}%
\special{pa 1572 637}%
\special{fp}%
\special{pa 1504 645}%
\special{pa 1500 646}%
\special{pa 1472 657}%
\special{pa 1448 672}%
\special{pa 1445 676}%
\special{fp}%
\special{pa 1412 736}%
\special{pa 1407 753}%
\special{pa 1401 787}%
\special{pa 1400 802}%
\special{fp}%
\special{pa 1396 871}%
\special{pa 1394 895}%
\special{pa 1389 927}%
\special{pa 1385 938}%
\special{fp}%
\special{pa 1344 990}%
\special{pa 1338 994}%
\special{pa 1307 1005}%
\special{pa 1280 1012}%
\special{fp}%
%
\special{pn 4}%
\special{sh 1}%
\special{ar 2080 1200 16 16 0 6.2831853}%
\put(20.4000,-13.2000){\makebox(0,0)[lb]{$*$}}%
%
\special{pn 8}%
\special{pa 1625 1564}%
\special{pa 1543 1568}%
\special{fp}%
\special{pa 1625 1564}%
\special{pa 1554 1606}%
\special{fp}%
%
\special{pn 8}%
\special{pa 1274 812}%
\special{pa 1192 809}%
\special{fp}%
\special{pa 1274 812}%
\special{pa 1200 848}%
\special{fp}%
%
\special{pn 8}%
\special{pa 1760 750}%
\special{pa 1765 832}%
\special{fp}%
\special{pa 1760 750}%
\special{pa 1803 820}%
\special{fp}%
\put(11.6000,-7.7200){\makebox(0,0)[lb]{$e_3$}}%
\put(18.3000,-8.7200){\makebox(0,0)[lb]{$e_2$}}%
\put(15.5000,-17.2200){\makebox(0,0)[lb]{$e_1$}}%
%
\special{pn 8}%
\special{pa 3400 495}%
\special{pa 4200 495}%
\special{pa 4600 1188}%
\special{pa 4200 1881}%
\special{pa 3400 1881}%
\special{pa 3000 1188}%
\special{pa 3400 495}%
\special{pa 4200 495}%
\special{fp}%
%
\special{pn 8}%
\special{pa 3800 495}%
\special{pa 3720 475}%
\special{fp}%
\special{pa 3800 495}%
\special{pa 3720 515}%
\special{fp}%
%
\special{pn 8}%
\special{pa 3760 1881}%
\special{pa 3840 1861}%
\special{fp}%
\special{pa 3760 1881}%
\special{pa 3840 1901}%
\special{fp}%
%
\special{pn 8}%
\special{pa 3177 1491}%
\special{pa 3235 1550}%
\special{fp}%
\special{pa 3177 1491}%
\special{pa 3200 1570}%
\special{fp}%
%
\special{pn 8}%
\special{pa 3233 783}%
\special{pa 3176 843}%
\special{fp}%
\special{pa 3233 783}%
\special{pa 3210 863}%
\special{fp}%
%
\special{pn 8}%
\special{pa 4439 906}%
\special{pa 4416 826}%
\special{fp}%
\special{pa 4439 906}%
\special{pa 4382 847}%
\special{fp}%
%
\special{pn 8}%
\special{pa 4389 1552}%
\special{pa 4449 1494}%
\special{fp}%
\special{pa 4389 1552}%
\special{pa 4414 1474}%
\special{fp}%
%
\special{pn 4}%
\special{sh 1}%
\special{ar 3800 1190 16 16 0 6.2831853}%
\put(37.5400,-13.4400){\makebox(0,0)[lb]{$*$}}%
\put(30.2500,-16.9000){\makebox(0,0)[lb]{$e_1$}}%
\put(30.0000,-7.9000){\makebox(0,0)[lb]{$e_2$}}%
\put(37.0000,-3.9000){\makebox(0,0)[lb]{$e_3$}}%
\put(45.0000,-8.9000){\makebox(0,0)[lb]{$\overline{e}_1$}}%
\put(45.0000,-16.4000){\makebox(0,0)[lb]{$\overline{e}_2$}}%
\put(37.0000,-20.6500){\makebox(0,0)[lb]{$\overline{e}_3$}}%
%
\special{pn 8}%
\special{ar 1400 1580 80 220 1.5707963 4.7123890}%
%
\special{pn 8}%
\special{pn 8}%
\special{pa 1400 1360}%
\special{pa 1405 1360}%
\special{pa 1406 1361}%
\special{pa 1409 1361}%
\special{pa 1410 1362}%
\special{pa 1412 1362}%
\special{pa 1412 1363}%
\special{pa 1414 1363}%
\special{pa 1415 1364}%
\special{pa 1418 1365}%
\special{pa 1418 1366}%
\special{pa 1421 1367}%
\special{pa 1421 1368}%
\special{pa 1422 1368}%
\special{pa 1422 1369}%
\special{pa 1423 1369}%
\special{pa 1423 1370}%
\special{pa 1424 1370}%
\special{pa 1432 1378}%
\special{pa 1432 1379}%
\special{pa 1433 1379}%
\special{pa 1433 1380}%
\special{pa 1436 1383}%
\special{pa 1436 1384}%
\special{pa 1437 1384}%
\special{pa 1437 1386}%
\special{pa 1438 1386}%
\special{pa 1438 1387}%
\special{pa 1440 1389}%
\special{pa 1440 1390}%
\special{fp}%
\special{pa 1461 1439}%
\special{pa 1462 1442}%
\special{pa 1463 1443}%
\special{pa 1463 1445}%
\special{pa 1464 1446}%
\special{pa 1464 1450}%
\special{pa 1465 1451}%
\special{pa 1465 1453}%
\special{pa 1466 1454}%
\special{pa 1466 1458}%
\special{pa 1467 1459}%
\special{pa 1467 1461}%
\special{pa 1468 1462}%
\special{pa 1469 1471}%
\special{pa 1470 1472}%
\special{pa 1470 1476}%
\special{pa 1471 1477}%
\special{pa 1471 1480}%
\special{pa 1472 1482}%
\special{pa 1473 1493}%
\special{pa 1474 1494}%
\special{fp}%
\special{pa 1479 1552}%
\special{pa 1479 1555}%
\special{pa 1480 1556}%
\special{pa 1480 1603}%
\special{pa 1479 1605}%
\special{pa 1479 1610}%
\special{fp}%
\special{pa 1473 1668}%
\special{pa 1473 1668}%
\special{pa 1473 1673}%
\special{pa 1472 1674}%
\special{pa 1471 1683}%
\special{pa 1470 1684}%
\special{pa 1470 1689}%
\special{pa 1469 1690}%
\special{pa 1468 1697}%
\special{pa 1467 1698}%
\special{pa 1467 1702}%
\special{pa 1466 1703}%
\special{pa 1466 1705}%
\special{pa 1465 1707}%
\special{pa 1465 1710}%
\special{pa 1464 1711}%
\special{pa 1464 1713}%
\special{pa 1463 1714}%
\special{pa 1463 1717}%
\special{pa 1462 1718}%
\special{pa 1461 1722}%
\special{fp}%
\special{pa 1439 1772}%
\special{pa 1439 1773}%
\special{pa 1438 1773}%
\special{pa 1438 1774}%
\special{pa 1436 1776}%
\special{pa 1436 1777}%
\special{pa 1435 1777}%
\special{pa 1434 1780}%
\special{pa 1433 1780}%
\special{pa 1433 1781}%
\special{pa 1432 1781}%
\special{pa 1432 1782}%
\special{pa 1427 1787}%
\special{pa 1427 1788}%
\special{pa 1426 1788}%
\special{pa 1423 1791}%
\special{pa 1422 1791}%
\special{pa 1422 1792}%
\special{pa 1421 1792}%
\special{pa 1421 1793}%
\special{pa 1418 1794}%
\special{pa 1418 1795}%
\special{pa 1416 1795}%
\special{pa 1416 1796}%
\special{pa 1414 1796}%
\special{pa 1414 1797}%
\special{pa 1412 1797}%
\special{pa 1412 1798}%
\special{pa 1410 1798}%
\special{pa 1409 1799}%
\special{pa 1406 1799}%
\special{pa 1405 1800}%
\special{pa 1400 1800}%
\special{fp}%
%
\special{pn 8}%
\special{pa 1350 1749}%
\special{pa 1345 1667}%
\special{fp}%
\special{pa 1350 1749}%
\special{pa 1307 1679}%
\special{fp}%
\put(11.8000,-17.4800){\makebox(0,0)[lb]{$\hat{e}_1$}}%
%
\special{pn 8}%
\special{pa 3300 1702}%
\special{pa 4490 1002}%
\special{fp}%
%
\special{pn 8}%
\special{pa 3722 1453}%
\special{pa 3645 1479}%
\special{fp}%
\special{pa 3722 1453}%
\special{pa 3666 1513}%
\special{fp}%
\put(37.2200,-16.1300){\makebox(0,0)[lb]{$\hat{e}_1$}}%
\end{picture}}%
\caption{Homology marking: the loop $\hat{e}_1$ is dual to $e_1$.}
\label{fig:hmarking}
\end{figure}

The homology marking satisfies the following properties:
\begin{enumerate}
\item
For any oriented edge $e$ of $G$, we have $\mu(\overline{e}) = - \mu(e)$.
\item
If $v$ is a vertex of $G$ and $e_1,\ldots,e_n$ are the oriented edges of $G$ pointing toward $v$, then
$
\sum_{i=1}^n \mu(e_i) = 0
$.
\end{enumerate}

\begin{example} \label{ex:hmarking}
\begin{enumerate}
\item
If $v$ is a trivalent vertex of $G$ and the three oriented edges pointing toward $v$ are arranged as in Figure \ref{fig:abc}, we have
\[
\mu(a_v) + \mu(b_v) + \mu(c_v) = 0.
\]
\item
In the notation of Figure \ref{fig:flip}, we have
\[
\mu(a) + \mu(b) + \mu(c) + \mu(d) = 0.
\]
\end{enumerate}
\end{example}

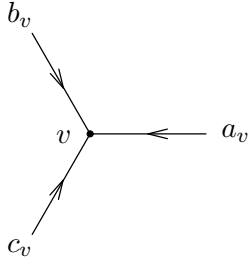
\begin{figure}[ht]
  \centering
{\unitlength 0.1in%
\begin{picture}(11.1000,12.3000)(11.7000,-15.2500)%
%
\special{pn 8}%
\special{pa 1600 1000}%
\special{pa 2200 1000}%
\special{fp}%
%
\special{pn 8}%
\special{pa 1600 1000}%
\special{pa 1300 475}%
\special{fp}%
%
\special{pn 8}%
\special{pa 1600 1000}%
\special{pa 1300 1525}%
\special{fp}%
%
\special{pn 8}%
\special{pa 1900 1000}%
\special{pa 2000 1025}%
\special{fp}%
\special{pa 1900 1000}%
\special{pa 2000 975}%
\special{fp}%
%
\special{pn 8}%
\special{pa 1462 762}%
\special{pa 1438 662}%
\special{fp}%
\special{pa 1462 762}%
\special{pa 1394 685}%
\special{fp}%
%
\special{pn 8}%
\special{pa 1460 1242}%
\special{pa 1436 1342}%
\special{fp}%
\special{pa 1460 1242}%
\special{pa 1392 1319}%
\special{fp}%
%
\special{pn 4}%
\special{sh 1}%
\special{ar 1600 1000 16 16 0 6.2831853}%
\put(14.2500,-10.4500){\makebox(0,0)[lb]{$v$}}%
\put(22.8000,-10.4000){\makebox(0,0)[lb]{$a_v$}}%
\put(11.7000,-4.4000){\makebox(0,0)[lb]{$b_v$}}%
\put(11.7000,-16.4000){\makebox(0,0)[lb]{$c_v$}}%
\end{picture}}%
\caption{Oriented edges pointing toward $v$.}
\label{fig:abc}
\end{figure}

\subsection{Penner's cocycle}

We keep the notation in Figure \ref{fig:flip}.
For the flip $W_e$, we set
\[
s(W_e) := (\mu(a) \wedge \mu(c)) \cdot (\mu(b) \wedge \mu(d))
\in S^2(\wedge^2 H).
\]
Here, $\wedge^2 H$ is the second exterior power of $H$, and $S^2(\wedge^2 H)$ is the second symmetric power of $\wedge^2 H$.
It is easy to see that $s(W_{e'}) = - s(W_e)$.
Since flips correspond to oriented $1$-cells in $\widehat{\mathcal{G}}$, we can regard $s$ as a $1$-cochain of $\widehat{\mathcal{G}}$ with values in $S^2 (\wedge^2 H)$.
It is in fact a $1$-cocycle:
\[
s \in Z^1( \widehat{\mathcal{G}}; S^2 (\wedge^2 H)).
\]
This means that $s$ respects the three types of relations among flips.
The only nontrivial relation one has to check is the pentagon relation, and we must have
\[
s(W_f) + s(W_{g_1}) + s(W_{f_2}) + s(W_{g_3}) + s(W_{f_4}) = 0
\]
where we use the notation in Figure \ref{fig:pentagon}.
A proof of this can be found in \cite[\S 3, Theorem 5]{KPT}.

The cocycle $s$ was introduced by Penner in \cite[Chapter 6, Remark 2.8]{PenBook}.

\section{Results} \label{sec:results}

First we fix some conventions in \S \ref{subsec:convention}.
Then we state and prove our results.

\subsection{Mapping class group and tensor modules of $H$} \label{subsec:convention}

Let $\mathcal{M}_{g,*}$ be the mapping class group of the punctured surface $\Sigma$.
Namely, $\mathcal{M}_{g,*}$ is the group of orientation preserving homeomorphisms of $\Sigma$ that preserve $*$ modulo isotopies that fix $*$.
The group $\mathcal{M}_{g,*}$ acts naturally as cellular automorphisms on the fatgraph complex $\widehat{\mathcal{G}}$.
It acts also on $H=H_1(\Sigma; \mathbb{Z})$ and hence on tensor modules of $H$ such as $\wedge^2 H$ and $S^2(\wedge^2 H)$.

The cocycle $s$ is $\mathcal{M}_{g,*}$-equivariant in the sense that
\[
\varphi\cdot s(W_{e}) = s(W_{\varphi(e)})
\]
for any flip $W_e$ and $\varphi \in \mathcal{M}_{g,*}$.

We regard $\wedge^2 H$ as a submodule of $H^{\otimes 2}$ defined as the image of the homomorphism
\[
H^{\otimes 2} \longrightarrow H^{\otimes 2}, \quad
x\otimes y \mapsto x\wedge y:= x\otimes y - y\otimes x,
\]
and $S^2 (\wedge^2 H)$ as a submodule of $\wedge^2 H \otimes \wedge^2 H$ defined as the image of the homomorphism
\[
\wedge^2 H \otimes \wedge^2H \longrightarrow \wedge^2 H \otimes \wedge^2H, \quad
u\otimes v \mapsto u\cdot v := u\otimes v + v\otimes u.
\]
Hence we can think of $S^2(\wedge^2 H)$ as a submodule of $H^{\otimes 2} \otimes H^{\otimes 2} = H^{\otimes 4}$.

The first homology group $H= H_1(\Sigma; \mathbb{Z})$ is equipped with a non-degenerate skew-symmetric bilinear form
\[
H^{\otimes 2} \longrightarrow \mathbb{Z},
\quad
x\otimes y \mapsto (x \cdot y)
\]
called the \emph{intersection form}.
Note that it sends
\begin{equation}
\label{ex:2xy}
x\wedge y  \mapsto 2 (x\cdot y)
\end{equation}
for any $x,y \in H$.

\subsection{A cobounding cochain for Penner's cocycle} \label{subsec:cobounding}

Let $G$ be a trivalent fatgraph spine of $\Sigma$ and let $v$ be a trivalent vertex of $G$.
In the notation in Figure \ref{fig:abc}, it holds that
\[
\mu(a_v) \wedge \mu(b_v) = \mu(b_v) \wedge \mu(c_v) = \mu(c_v) \wedge \mu(a_v)
\]
since $\mu(a_v) + \mu(b_v) + \mu(c_v) = 0$ as explained in Example \ref{ex:hmarking} (1).
We put
\[
\eta_v := \mu(a_v) \wedge \mu(b_v) \in \wedge^2 H.
\]

\begin{dfn}
\[
\xi_G := \frac{1}{2} \sum_v \eta_v \cdot \eta_v =
\sum_v \eta_v \otimes \eta_v.
\]
Here, the sum is taken over all trivalent vertices of $G$.
\end{dfn}

A priori, $\xi_G$ is an element of $(1/2) S^2 (\wedge^2 H)$.
The collection $\xi = \{ \xi_G\}_G$ can be regarded as a $0$-cochain of $\widehat{\mathcal{G}}$ with values in $(1/2)S^2 (\wedge^2 H)$.

\begin{thm} \label{thm:main}
For any fatgraph spine $G$ of $\Sigma$, we have $\xi_G \in S^2(\wedge^2 H)$.
The $0$-cochain $\xi\in C^0(\widehat{\mathcal{G}}; S^2(\wedge^2 H))$ cobounds the $1$-cocycle $s$ and is $\mathcal{M}_{g,*}$-equivariant in the sense that $\varphi \cdot \xi_G = \xi_{\varphi(G)}$ for any $G$ and $\varphi \in \mathcal{M}_{g,*}$.
\end{thm}

\begin{proof}
That $\xi$ is $\mathcal{M}_{g,*}$-equivariant is clear from construction.

Next we prove that $\delta \xi = s$.
We need to show that $\xi_{G'} - \xi_G = s(W_e)$, where we use the notation in Figure \ref{fig:flip}.
For simplicity, we write $\mu(a)=a$, etc.
Note that there is a canonical bijection between the set of vertices of $G$ that are different from the endpoints of $e$ and the set defined in the same way for $G'$, and this bijection preserves the value of $\eta$.
Thus we have
\[
2 (\xi_{G'} - \xi_G) = (b\wedge c)^2 + (d\wedge a)^2 - (a\wedge b)^2 - (c\wedge d)^2,
\]
where $(b\wedge c)^2 = (b\wedge c)\cdot (b\wedge c)$, etc.
The right hand side is equal to
\[
(b\wedge c + a\wedge b) \cdot (b\wedge c - a\wedge b)
+ (d\wedge a + c\wedge d) \cdot (d\wedge a -c\wedge d).
\]
By Example \ref{ex:hmarking} (2), we have $a+b+c+d=0$.
Hence we have
\[
b\wedge c - a\wedge b = b\wedge (a+c)
= -b\wedge (b+d) = -b\wedge d,
\]
and similarly $d\wedge a - c\wedge d = b\wedge d$.
Therefore,
\begin{align*}
2 (\xi_{G'} - \xi_G) & =
(- b\wedge c - a \wedge b + d\wedge a + c\wedge d) \cdot (b\wedge d) \\
& = ( -a \wedge (b+d) -(b+d) \wedge c) \cdot (b\wedge d) \\
& = (a\wedge (a+c) + (a+c) \wedge c) \cdot (b\wedge d) \\
& = 2 (a\wedge c) \cdot (b\wedge d) \\
& = 2s(W_e).
\end{align*}
In the third line, we have used again the relation $a + b + c + d = 0$.
Since $S^2(\wedge^2 H)$ is torsion free, we obtain $\xi_{G'}-\xi_G = s(W_e)$.

Finally we prove that $\xi_G \in S^2(\wedge^2 H)$ for any $G$.
Consider the trivalent fatgraph spine $G_0$ in Example \ref{ex:G0} below.
Then, as we will see, it holds that $\xi_{G_0} \in S^2(\wedge^2 H)$.
Let $G$ be any trivalent fatgraph spine of $\Sigma$.
Take a sequence of flips
\[
G_0 \overset{W_1}{\longrightarrow} G_1
\overset{W_2}{\longrightarrow} \cdots
\overset{W_m}{\longrightarrow} G_m = G
\]
connecting $G_0$ to $G$.
Then from $\delta \xi = s$ we obtain $\xi_G = \xi_{G_0} + \sum_{j=1}^m s(W_j)$, which shows that $\xi_G \in S^2(\wedge^2 H)$ as well.
\end{proof}

\begin{example} \label{ex:G0}
Let $G_0$ be the trivalent fatgraph spine as in the left part of Figure \ref{fig:G0sympbasis} and fix a symplectic basis $\{ a_i,b_i\}_{i=1}^g$ for $H$ as shown in the right part of the same figure.
Then, we have $\mu(e_i) = -b_i$, $\mu(e_i') = a_i$, and $\mu(e''_i) = b_i$.
Thus $\eta_{v_i} = \mu(e_i) \wedge \mu(e_i') = a_i \wedge b_i$ and $\eta_{v'_i} = \mu(e''_i) \wedge \mu(\overline{e'_i}) = a_i \wedge b_i$.
The value of $\eta$ on the other vertices are zero, and hence
\[
\xi_{G_0} = \sum_{i=1}^g (a_i \wedge b_i) \cdot (a_i \wedge b_i).
\]
In particular, we see that $\xi_{G_0} \in S^2(\wedge^2 H)$.

\begin{figure}[ht]
  \centering
\input{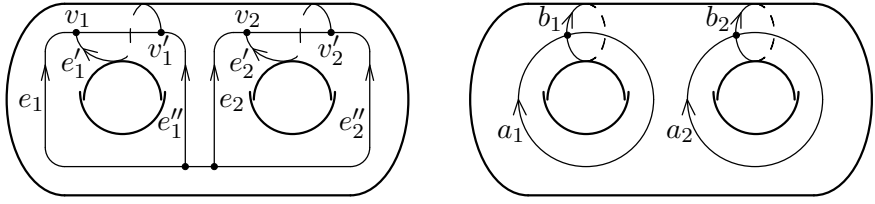}
\caption{The fatgraph spine $G_0$ and a symplectic basis ($g=2$).}
\label{fig:G0sympbasis}
\end{figure}
\end{example}

\begin{rem}
\begin{enumerate}
\item[(1)]
Since the module of $\mathcal{M}_{g,*}$-invariant tensors in $S^2(\wedge^2 H)$ is non-trivial (it is of rank $2$ if $g$ is at least $4$; see \cite[Lemma 4.1]{MorTaniguchi} for example), the $\mathcal{M}_{g,*}$-equivariant cobounding cochains for $s$ are not unique.
\item[(2)]
The cochain $\xi$ can be considered as a secondary invariant associated with the vanishing of the cohomology class $[s_G]$ which will be explained in the next subsection.
In \cite{K17}, another secondary invariant associated with the vanishing of a certain cohomology class in $H^1(\mathcal{M}_{g,*};H)$ was studied.
In that case, the uniqueness of the secondary invariant holds since there are no non-trivial $\mathcal{M}_{g,*}$-invariant elements in $H$.
\end{enumerate}
\end{rem}

\subsection{Cohomology vanishing}

Recall that the cocycle $s$ is $\mathcal{M}_{g,*}$-equivariant.
In this situation, given a trivalent fatgraph spine $G$ of $\Sigma$ one can construct a twisted $1$-cocycle
\[
s_G \colon \mathcal{M}_{g,*} \longrightarrow S^2(\wedge^2 H).
\]
Explicitly, for $\varphi \in \mathcal{M}_{g,*}$ we take a sequence of flips
\[
G = G_0 \overset{W_1}{\longrightarrow} G_1
\overset{W_2}{\longrightarrow} \cdots
\overset{W_m}{\longrightarrow} G_m = \varphi(G)
\]
which connects $G$ to $\varphi(G)$, and set
\[
s_G(\varphi) := \sum_{i=1}^m s(W_i).
\]
The twisted cohomology class
\[
[s_G] \in H^1(\mathcal{M}_{g,*}; S^2(\wedge^2 H))
\]
is independent of the choice of $G$.
A more detailed and general construction is explained in \cite[the proof of Theorem 5]{KPT} and in \cite[\S 2]{K17}.
This combinatorial construction of cocycles of $\mathcal{M}_{g,*}$ from $\mathcal{M}_{g,*}$-equivariant cocycles of $\widehat{\mathcal{G}}$ originally dates back to the work of Morita and Penner \cite{MP}.

In \cite[\S 3]{KPT}, it was shown that $[s_G]$ vanishes over the \emph{rational} coefficients.
As a consequence of Theorem \ref{thm:main}, we see that the vanishing of $[s_G]$ holds over the \emph{integral} coefficients.

\begin{cor}
$[s_G]= 0 \in H^1(\mathcal{M}_{g,*}; S^2(\wedge^2 H))$.
\end{cor}

\begin{proof}
Let $\varphi \in \mathcal{M}_{g,*}$.
By construction and  by Theorem \ref{thm:main}, we have $s_G(\varphi) = \xi_{\varphi(G)} - \xi_G$.
Since the $0$-cochain $\xi$ is $\mathcal{M}_{g,*}$-equivariant, the right hand side is equal to $\varphi \cdot \xi_G - \xi_G$.
Hence $s_G = \delta \xi_G$ and $[s_G]=0$.
\end{proof}

\subsection{Contractions of Penner's cocycle}

Using the intersection form on $H$, we introduce the contraction maps
\[
{\rm Cont}_{1,2}\colon H^{\otimes 4} \longrightarrow H^{\otimes 2},
\quad
x\otimes y \otimes z \otimes w \mapsto
(x\cdot y)\, z\otimes w
\]
and
\[
{\rm Cont}_{1,3}\colon H^{\otimes 4} \longrightarrow H^{\otimes 2},
\quad
x\otimes y \otimes z \otimes w \mapsto
(x\cdot z)\, y\otimes w.
\]
Their restrictions to $S^2(\wedge^2 H) \subset H^{\otimes 4}$ take values in $\wedge^2 H$.
We consider $1$-cocycles
\[
s':= {\rm Cont}_{1,2}\circ s
\quad
\text{and}
\quad
s'':= {\rm Cont}_{1,3} \circ s.
\]
Explicitly, their values on the flip $W_e$ are given as follows:
\begin{align*}
s'(W_e) &= 2 \left( (a\cdot c)\, b\wedge d + (b\cdot d)\, a \wedge c \right), \\
s''(W_e) &=
(a\cdot b)\, c\wedge d + (a\cdot d)\, b\wedge c
+ (c\cdot d)\, a\wedge b + (b\cdot c)\, a\wedge d.
\end{align*}
Here and in the rest of this section, we simply write $\mu(a)=a$, etc.

\begin{prop} \label{prop:s2s}
We have $s' = 2 s'' \in Z^1(\widehat{\mathcal{G}}; \wedge^2 H)$.
\end{prop}

\begin{proof}
We work with Figure \ref{fig:flip} and prove the equality $s'(W_e) = 2s''(W_e)$.
We need to take into account possible cyclic orderings among oriented edges $a,b,c,d$.
By changing the role of $a$ and $c$ and the role of $G$ and $G'$, it is sufficient to consider the following six cases:

\begin{center}
\begin{tabular}{rrrr}
(i) & $a \prec b \prec c \prec d \prec a$;
& \quad (ii) & $a\prec b \prec d \prec c \prec a$; \\
(iii) & $a\prec c \prec b \prec d \prec a$;
& \quad (iv) & $a\prec c \prec d \prec b \prec a$; \\
(v) & $a\prec d \prec b \prec c \prec a$;
& \quad (vi) & $a\prec d \prec c \prec b \prec a$.
\end{tabular}
\end{center}
We only consider the cases (i) and (ii), the other cases being similar.

Case (i).
From Figure \ref{fig:twopatterns}, we see that all the intersection numbers $(a\cdot c), (b\cdot d), (a\cdot b), (a\cdot d), (c\cdot d)$, and $(b\cdot c)$ are zero, and thus $s'(W_e) = s''(W_e)=0$.

Case (ii).
From Figure \ref{fig:twopatterns}, we have
$
(a\cdot c) = 1, (b\cdot d) = 0,
(a\cdot b) = 0, (a\cdot d) = -1,
(c\cdot d) = +1, (b\cdot c)=0
$.
Therefore, $s'(W_e) = 2 b\wedge d$, and
\[
s''(W_e)  = -b\wedge c + a \wedge b
= -b \wedge (a+c) 
= b\wedge (b+d) 
= b\wedge d.
\]
Thus $s'(W_e) = 2 s''(W_e)$, as required.
\end{proof}

\begin{figure}[ht]
  \centering
\input{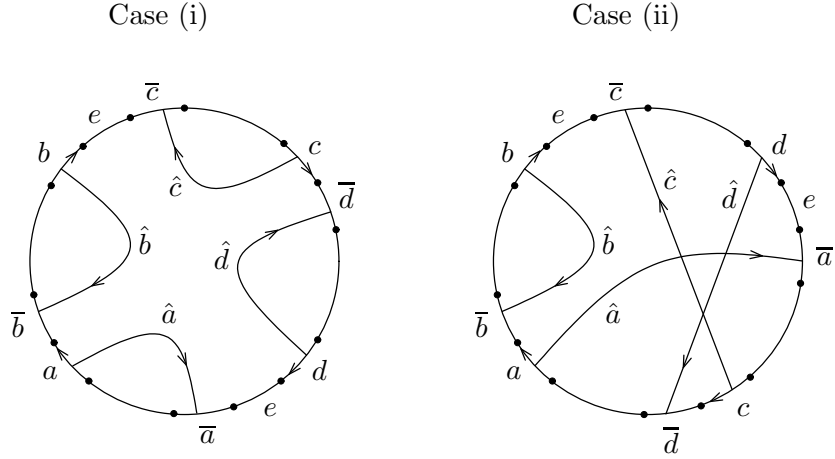}
\caption{Case (i) and Case (ii).}
\label{fig:twopatterns}
\end{figure}

\subsection{Cobounding cochains for $s'$} \label{subsec:cocos'}

By Theorem \ref{thm:main},
\[
\xi':={\rm Cont}_{1,2}\circ \xi \in C^0(\widehat{\mathcal{G}}; \wedge^2 H)
\]
is an $\mathcal{M}_{g,*}$-equivariant $0$-cochain which cobounds the $1$-cocycle $s'$.
Let $G$ be a trivalent fatgraph spine of $\Sigma$.
To describe $\xi'_G$ explicitly, we recall the following definition from \cite[\S 3]{K17}.

\begin{dfn}
Let $v$ be a trivalent vertex of $G$ and use the notation in Figure \ref{fig:abc}.
The vertex $v$ is called of \emph{type 1} if $a_v \prec b_v \prec c_v \prec a_v$, and of \emph{type 2} if $a_v \prec c_v \prec b_v \prec a_v$.
\end{dfn}

Observe the following:
\[
\begin{cases}
(a_v\cdot b_v) = (b_v \cdot c_v) = (c_v \cdot a_v)=0 &
\text{if $v$ is of type 1;} \\
(a_v \cdot b_v) = (b_v \cdot c_v) = (c_v \cdot a_v) = 1 &
\text{if $v$ is of type 2.}
\end{cases}
\]
By \eqref{ex:2xy} and $\xi_G = \sum_v \eta_v \otimes \eta_v = \sum_v (a_v \wedge b_v) \otimes (a_v \wedge b_v)$, we have
\begin{equation}
\label{eq:xi1}
\xi'_G = 2 \sum_v (a_v\cdot b_v)\, a_v \wedge b_v
=2 \sum_{v:\, \text{type 2}} a_v \wedge b_v.
\end{equation}

The $\mathcal{M}_{g,*}$-equivariant cobounding cochains for the $1$-cocycle $s'$ are not unique.
Let us consider another  $\mathcal{M}_{g,*}$-equivariant cochain $\widetilde{\xi} = \{ \widetilde{\xi}_G\}_G \in C^0(\widehat{\mathcal{G}}; \wedge^2 H) $ defined by
\begin{equation}
\label{eq:xi2}
\widetilde{\xi}_G :=
-\sum_{v:\, \text{type 1}} a_v \wedge b_v
+\sum_{v:\, \text{type 2}} a_v \wedge b_v.
\end{equation}

\begin{lem}
$\delta \widetilde{\xi} = s'$.
\end{lem}

\begin{proof}
We work with Figure \ref{fig:flip} and prove $\widetilde{\xi}_{G'} - \widetilde{\xi}_G = s'(W_e)$.
We consider the six cases in the proof of Proposition \ref{prop:s2s}.
We only treat the first two cases, the other cases being similar.

Case (i).
All the trivalent vertices in Figure \ref{fig:flip} are of type 1, and hence
\[
\widetilde{\xi}_{G'} - \widetilde{\xi}_G
= -b\wedge c-d\wedge a -(-a\wedge b -c\wedge d)
=(a+c) \wedge (b+d) =0.
\]
Since $s'(W_e) =0$ as shown in the proof of Proposition \ref{prop:s2s}, we obtain $\widetilde{\xi}_{G'} - \widetilde{\xi}_G = s'(W_e)$.

Case (ii). Suppose $a\prec b \prec d \prec c \prec a$.
Let $v_1$ be the vertex of $G$ adjacent to $a,b,e$ and $v_2$ the one adjacent to $c,d,e$.
Similarly, let $v'_1$ be the vertex of $G'$ adjacent to $b,c,e'$ and $v'_2$ the one adjacent to $a,d,e'$.
Then $v_1$ and $v'_1$ are of type 1, and $v_2$ and $v'_2$ are of type 2.
Thus we compute
\begin{align*}
\widetilde{\xi}_{G'} - \widetilde{\xi}_G
&= -b\wedge c + d\wedge a - (-a\wedge b + c\wedge d) \\
&= -b \wedge (a+c) -(a+c) \wedge d \\
&= b\wedge (b+d) + (b+d) \wedge d \\
&= 2 b\wedge d.
\end{align*}
As was shown in the proof of Proposition \ref{prop:s2s}, we have $s'(W_e) = 2b\wedge d$.
This completes the proof.
\end{proof}

\subsection{A combinatorial formula for the symplectic form}

Let $\omega \in \wedge^2 H$ be the \emph{symplectic form} on $H=H_1(\Sigma;\mathbb{Z})$.
This is the element corresponding to the identity $1_H \in {\rm Hom}_{\mathbb{Z}}(H,H)$ through the isomorphism
\[
H \otimes H \cong {\rm Hom}_{\mathbb{Z}}(H,H),
\quad
x\otimes y \mapsto ( z \mapsto (x\cdot z) y)
\]
defined by the intersection form, where $x,y,z \in H$.
Explicitly, if $\{ a_i,b_i\}_{i=1}^g$ is a symplectic basis as in Example \ref{ex:G0}, then
\[
\omega = \sum_{i=1}^g a_i \wedge b_i.
\]
It is a generator of the module of $\mathcal{M}_{g,*}$-invariants in $\wedge^2 H$.

\begin{thm} \label{thm:symplectic_form}
Let $G$ be a trivalent fatgraph spine of $\Sigma$ and keep the notations above.
Then we have
\[
\omega = \frac{1}{2} \sum_v \eta_v = \frac{1}{2} \sum_v a_v \wedge b_v,
\]
where the sum is taken over all trivalent vertices of $G$.
\end{thm}

\begin{proof}
As we have seen in \S \ref{subsec:cocos'}, $\xi'$ and $\widetilde{\xi}$ are cobounding cochains for the same $1$-cocycle $s'$.
Thus their difference is a $0$-cocycle:
\[
\zeta:= \xi' - \widetilde{\xi} \in Z^0(\widehat{\mathcal{G}}; \wedge^2 H).
\]
This means that the value $\zeta_G$ is invariant under flips and hence is independent of $G$.
Let us denote this value by $\zeta_0 \in \wedge^2 H$.
Then $\zeta_0$ is $\mathcal{M}_{g,*}$-invariant since $\zeta$ is $\mathcal{M}_{g,*}$-equivariant.
Thus we can write $\zeta_0 = m \omega$ for some $m\in \mathbb{Z}$.

To determine $m$, we use the intersection form $H^{\otimes 2} \to \mathbb{Z}$ restricted to $\wedge^2 H$.
On the one hand, from \eqref{eq:xi1} and \eqref{eq:xi2}
we see that for any $G$
\[
\zeta_0 = \xi'_G - \widetilde{\xi}_G = \sum_v a_v \wedge b_v.
\]
As shown in \cite[Proposition 3.3]{K17}, there are $2g$ vertices of type 2.
(In \cite{K17} one considers fatgraph spines \emph{with tail}, but notice that the trivalent vertex adjacent to the tail is of type 1.)
Therefore, the value of the intersection form for $\zeta_0$ is $2\sum_v (a_v\cdot b_v) = 2\, \sharp \{ \text{vertices of type $2$} \} = 4g$.
On the other hand, the value of the intersection form for $\omega$ is $2g$.
Therefore, we conclude that $m=2$ and $\zeta_0 = 2\omega$.
This completes the proof.
\end{proof}

\begin{rem}
\begin{enumerate}
\item
Our formula for $\omega$ works for a fatgraph spine with tail as well.
We note that the vertex next to the tail does not contribute to the sum $(1/2)\sum_v \eta_v$ since the homology marking of the tail is trivial.
\item
In \cite{Ma}, Massuyeau constructed a 2-chain $Z_G$ in the normalized bar complex of $\pi:= \pi_1(\Sigma_{g,1})$, where $\Sigma_{g,1}$ is a once bordered surface of genus $g$ and $G\subset \Sigma_{g,1}$ is a fatgraph spine with tail.
By the canonical homomorphism $\pi \to H_1(\Sigma_{g,1}) \cong H$, the $2$-chain $Z_G$ is mapped to a normalized $2$-cocycle $Z'_G$ of $H$.
Massuyeau points out that $Z'_G$ has a similar expression to our bivector $(1/2)\sum_v \eta_v$ and its homology class in $H_2(H) \cong \Lambda^2 H$ coincides with $\omega$.
\end{enumerate}
\end{rem}

\end{document}